\documentclass[a4paper]{article}

\usepackage[english]{babel}
\usepackage{amsmath, amsthm, amsfonts}
\usepackage{amssymb}
\usepackage{graphics,color}
\usepackage{pstricks}
\usepackage{pstricks-add,pst-node}
\usepackage{subfigure}

\RequirePackage[latin1]{inputenc}    

\newtheorem{thm}{Theorem}
\newtheorem{thm*}{Theorem*}
\newtheorem{lem}[thm]{Lemma}
\newtheorem{cor}[thm]{Corollary}
\newtheorem{conj}[thm]{Conjecture}
\newtheorem{prop}[thm]{Proposition}
\newtheorem{rem}[thm]{Remark}

\DeclareMathOperator{\cro}{cr}

\begin{document}


\def\TitreRapport{
    Drawing disconnected graphs on the Klein bottle
 }

\def\NomsAuteurs{
    Laurent Beaudou${}^1$, Antoine Gerbaud${}^1$, Roland Grappe${}^2$ and Fr\'ed\'eric Palesi${}^1$
}

\def\DateRapport{
    14 mars 2008
}

\def\AdressePostale{
    ${}^1$Institut Fourier\\
    100, rue des Maths\\
    38\,402 St-Martin d'H\`eres -- FRANCE\\
    \medskip
    ${}^2$Laboratoire G-Scop\\
    46 avenue F\'{e}lix Viallet\\
    38\,000 Grenoble -- FRANCE.
}

\def\AdresseMail{
  laurent.beaudou@ujf-grenoble.fr\\
  antoine.gerbaud@ujf-grenoble.fr\\
  roland.grappe@g-scop.inpg.fr\\
  frederic.palesi@ujf-grenoble.fr
}

\def\ResumeAnglais{
We prove that two disjoint graphs must always be drawn separately on the Klein bottle in order to minimize the crossing number of the whole drawing.
    \\[1mm]
    {\bf Keywords: } Klein bottle, topological graph theory, crossing number.
}

\def\ResumeFrancais{
Dans ce rapport, nous prouvons que deux graphes disjoints doivent toujours \^etre dessin\'es s\'epar\'ement sur la bouteille de Klein lorsque le nombre de croisements du dessin est minimal.
    \\[1mm]
    {\bf Mots-cl\'es: } bouteille de Klein, graphes topologiques, croisements.
}



\def\Sfill#1{\vspace*{\stretch{#1}}}

\thispagestyle{empty}
\begin{center}
{\bf \Huge Institut Fourier} \\ \Sfill{5}
{\Large Institut Fourier}\\[1mm] \Sfill{5}
{\normalsize Unit\'e Mixte de Recherche 5582}\\
{\normalsize CNRS -- Universit\'e Joseph Fourier}\\ \Sfill{10}
\begin{tabular}{|p{11cm}|}
\hline 
\begin{center}
\vspace{4mm}
\baselineskip=1.3\normalbaselineskip
{\bf\Large \TitreRapport}\\[10mm]
{\bf\large \NomsAuteurs}\\[10mm]
\baselineskip=\normalbaselineskip 
{\bf \DateRapport}
\end{center}
\\[4mm] \hline
\end{tabular} \\ \Sfill{10}
\end{center}
\newpage

\thispagestyle{empty}
\begin{center}
\baselineskip=1.3\normalbaselineskip
{\bf\Large \TitreRapport}\\[8mm]
{\bf\large \NomsAuteurs}\\[4mm]
\baselineskip=\normalbaselineskip 
\AdressePostale\\[4mm]
\AdresseMail\\[10mm]
{\bf Abstract/R\'esum\'e}
\end{center}
\ResumeAnglais \\ [4mm]
\ResumeFrancais
\newpage


\section*{Introduction}

All graphs in this paper are finite, undirected and without loops. A path of $G$ is a sequence of vertices $v_0 , \ldots , v_{k}$ of $G$ such that for each integer $i$ between $1$ and $k-1$, $v_iv_{i+1}$ is an edge of $G$ and all edges are distinct. A circuit of $G$ is a path $v_0,\ldots,v_k$ such that $v_0=v_k$. A graph with a circuit that visits each of its edges exactly once is called eulerian. A graph is connected if for every pair of vertices $u$ and $v$ there is a path $v_0,\ldots,v_k$ such that $v_0=u$ and $v_k=v$. We refer to \cite{bollobas_98} for an introduction to graph theory. 

A surface is a two-dimensional manifold, with or without boundary. According to \cite{brahana_23}, there are two infinite classes of compact connected surfaces without boundary: the orientable surfaces homeomorphic to a sphere with handles attached, and the non-orientable surfaces homeomorphic to a connected sum of projective planes. For an orientable surface, the number of handles is called the orientable genus. For a non-orientable surface, the number of projective planes is called the non-orientable genus. The non-orientable surfaces of genus $1$ and $2$ are the projective plane and the Klein bottle, respectively. Formal definitions of these surfaces can be found in \cite{stillwell_80}. 

Every curve considered throughout this paper is undirected and we do not distinguish between a curve and its image. A drawing of a graph $G$ on a surface $\Sigma$ is a representation $\Psi$ of $G$ on $\Sigma$ where vertices are distinct points of $\Sigma$, and edges are curves of $\Sigma$ joining the points corresponding to their endvertices. A drawing is proper if edges are simple curves without vertices of the graph in their interiors. A crossing is a transversal intersection of two curves on $\Sigma$. In this paper, we restrict our attention to proper drawings where two incident edges do not cross each other, two non-incident edges cross at most once and no more than two edges cross at a single point. The crossing number of a drawing $\Psi$, denoted by $\cro(\Psi)$, is the number of crossings between each pair of curves in $\Psi$. The crossing number of a graph $G$ on a surface $\Sigma$ is the minimum crossing number among all drawing of $G$ on $\Sigma$. A drawing that achieves the crossing number of a graph is said optimal. A drawing with no crossing is an embedding. For background material about topological graph theory, the reader can refer to \cite{mohar_01}. 

The crossing number of a graph on a surface leads to many unsolved problems, see \cite{erdos_73,toth_00}. DeVos, Mohar and Samal conjectured the following in \cite{devos_08}.

\begin{conj}
\label{conj}
  Let $G$ be the disjoint union of two connected graphs $H$ and $K$ and let $\Sigma$ be a surface. For every optimal drawing of $G$ on $\Sigma$, the restrictions to $H$ and $K$ do not intersect.
\end{conj}

This conjecture is obviously true for the sphere or equivalently for the Euclidean plane. It was announced proved for the projective plane in \cite{devos_08}. The problem remains open in the general case. In this paper, we prove that Conjecture \ref{conj} holds if $\Sigma$ is the Klein bottle.

\begin{thm}
\label{thm:main}
  Let $G$ be the disjoint union of two connected graphs $H$ and $K$. For every optimal drawing of $G$ on the Klein bottle, the restrictions to $H$ and $K$ do not cross.
\end{thm}

We introduce the following notations. A closed curve is one-sided if its neighborhood is a Mˆbius strip, two-sided otherwise. There exist two non freely homotopic one-sided simple curves $a$ and $b$ on the Klein bottle, a two-sided simple curve $m$ that cuts open the Klein bottle into a cylinder, and a two-sided simple curve $e$ that separates the Klein bottle into two Mˆbius strips. A closed curve not contractible is called essential. According to Negami in \cite{negami_97}, each essential simple closed curve on the Klein bottle is freely homotopic to either $a$, $b$, $m$ or $e$. 

For each curve $c$ on $\Sigma$, $[c]$ denotes the set of curves freely homotopic to $c$. For each couple of curves $(c,d)$, $\cro([c],[d])$ denotes the minimum number of crossings, counting multiplicities, taken over all couples of $[c]\times [d]$. Let $c$ be a curve on $\Sigma$ and $I$ a collection of curves. The number of crossings between $c$ and $I$ is denoted by $\cro(c,I)$. The minimum of $\cro(c',I)$ taken over all curves $c'$ in $[c]$ is denoted by $\cro([c],I)$. If $I$ is a drawing of a graph $G$, the minimum $\cro([c],I)$ is taken on the curves in $[c]$ that do not contain any vertex of $G$.

We define two relations on freely homotopy classes of closed curves on the Klein bottle. Two classes $[c]$ and $[d]$ are said to be orthogonal if $\cro([c],[d])\ge 1$, otherwise disjoint. These definitions slightly differ from those of Luo in \cite{luo_97}. Let $\Psi$ be a drawing on the Klein bottle. The circuits $c$ of $\Psi$ orthogonal to $[a]$ and disjoint from $[b]$ are called $a$-circuits. The circuits orthogonal to $[b]$ and disjoint from $[a]$ are called $b$-circuits. The circuits orthogonal to $[a]$ and $[b]$ are called $m$-circuits. Finally, the circuits orthogonal to $[m]$ and disjoint from $[a]$ and $[b]$ are called $e$-circuits. 

We will apply the following result.
\begin{thm}(De Graaf, Schrijver \cite{de_graaf_97})
\label{thm:de_graaf}
Let $\Psi$ be an embedding of an eulerian graph on a metrizable surface $\Sigma$. Then $\Psi$ can be decomposed into a collection of circuits $I$ such that for each closed curve $c$ on $\Sigma$, 
$$
\cro([c],\Psi)=\sum_{d \in I}\cro([c],[d]).
$$
\end{thm}
Decomposing a drawing $\Psi$ of a graph $G$ into a collection of circuits $I$ means that each edge of $G$ in $\Psi$ is visited by exactly one element of $I$ and by that element only once. Theorem \ref{thm:de_graaf} was first proved by Lins in \cite{lins_81} for the projective plane and for compact orientable surfaces by Schrijver in \cite{schrijver_91}. 

Schrijver proved in \cite{schrijver_89} an extension of Lins' result to the Klein bottle: the maximum number of pairwise edge-disjoint one-sided circuits equals the minimum number of edges intersecting all one-sided circuits. Using Theorem \ref{thm:de_graaf} and operations on circuits similar to the product studied for oriented surfaces in \cite{luo_97}, we give another expression of this number as follows.  

\begin{prop}
\label{pro1}
Let $\Psi$ be an embedding of an eulerian graph on the Klein bottle. Then the maximum number of pairwise edge-disjoint one-sided circuits equals 
$$
\min(\cro([a],\Psi)+\cro([b],\Psi),\cro([m],\Psi)).
$$
Moreover, we can decompose $\Psi$ into a collection of circuits $I$ that achieves the maximum number of one-sided circuits and such that the number of $m$-circuits in $I$ is
$$
\frac{1}{2} \left( \cro([a],\Psi)+\cro([b],\Psi) - \min(\cro([a],\Psi)+\cro([b],\Psi),\cro([m],\Psi)) \right).
$$
\end{prop}

Similarly, we express the maximum number of edge-disjoint $a$-circuits. 

\begin{prop}
\label{pro2}
  Let $\Psi$ be an embedding of an eulerian graph on the Klein bottle. Then the maximum number of edge-disjoint $a$-circuits equals 
$$
\min(\cro([a],\Psi),\cro([m],\Psi)).
$$
\end{prop}

Let us sketch the proof of Theorem \ref{thm:main}. Let $G$ be a disjoint union of two connected graphs $H$ and $K$. To prove Theorem \ref{thm:main}, we start with a drawing $\Psi$ of $G$ where $H$ and $K$ cross each other and we construct by topological surgery another drawing $\Psi'$ of $G$ where $H$ and $K$ do not. Proposition \ref{pro1} and Proposition \ref{pro2} provide lower bounds on the number of crossings between the drawings of two disjoint graphs on the Klein bottle. They allow us to show that the drawing $\Psi'$ has strictly fewer crossings than $\Psi$. 

This paper is organized as follows. Section \ref{Section:operation} deals with drawing graphs on surfaces of smaller genus. Section \ref{Section:pro} is devoted to the proof of Proposition \ref{pro1} and Section \ref{Section:pro2} to the proof of Proposition \ref{pro2}. Theorem \ref{thm:main} is proved in Section \ref{Section:proof}.

\section{Drawing graphs on surfaces of smaller genus}
\label{Section:operation}

In this section, starting with a drawing of a graph on a surface, we define new drawings of the same graph on surfaces of smaller genus. We compute the crossing numbers of these drawings. 

\subsection{Removing a crosscap}

Let $\Sigma$ be the non-orientable surface of genus $g$ and $\Sigma '$ be the non-orientable surface of genus $g-1$.  

\begin{prop}\label{switch}
Let $\Psi$ denote a drawing of a graph $G$ on $\Sigma$. Let $c$ be a simple closed one-sided curve on $\Sigma$ which does not contain any vertex of $G$. There is a drawing $\Psi'$ of $G$ on $\Sigma'$ such that
$$ \cro (\Psi') = \cro (\Psi) + \frac{\cro(c,\Psi)(\cro(c,\Psi) - 1)}{2}.$$
\end{prop}

\begin{proof}
We cut open $\Sigma$ along $c$ and we obtain the non-orientable surface $\Sigma '$ with one hole. We can glue a disk $D$ along the boundary component to obtain $\Sigma '$. Let $\Psi'$ be the drawing of $G$ defined by restricting $\Psi$ to $\Sigma \setminus c$ and redrawing the edges of $\Psi$ that crossed $c$ on $D$, crossing exactly once pairwise. The crossings of these edges add to the crossings of $\Psi$ to give the correct number of crossings of $\Psi'$ stated in Proposition \ref{switch}.
\end{proof}

The non-orientable surface of genus $g$ can be seen as a sphere with $g$ crosscaps attached. Attaching a crosscap to a surface $\Sigma$ means removing an open disk $D$ of $\Sigma$ and identifying opposite points on the boundary of $D$.

\begin{cor}
\label{cor:planproj}
Let $G$ be the disjoint union of two eulerian connected graphs $H$ and $K$. If $G$ has a drawing on the projective plane such that the restrictions to $H$ and $K$ are embeddings that cross each other, then we can find another drawing of $G$ on the projective plane with strictly fewer crossings such that the restrictions to $H$ and $K$ do not cross. 
\end{cor}

\begin{proof}
Let $\Psi$ be a drawing of $G$ on the projective plane such that the restriction $\Psi_H$ to $H$ and the restriction $\Psi_K$ to $K$ are embeddings. 

All one-sided simple essential closed curves on the projective plane are freely homotopic. Let $c$ be such a curve. By a theorem of Lins \cite{lins_81}, the maximum number of edge-disjoint one-sided circuits of $\Psi_H$ and $\Psi_K$ are $\cro([c],\Psi_H)$ and $\cro([c],\Psi_K)$, respectively. We may assume that $\cro([c],\Psi_H)$ is smaller than $\cro([c],\Psi_K)$. 

Two one-sided circuits cross at least once. Hence, each one-sided circuit of $\Psi_H$ crosses each one-sided circuit of $\Psi_K$, and
$$
\cro(\Psi)\ge \cro([c],\Psi_H)\times \cro([c],\Psi_K).
$$

Let $c'$ be an one-sided closed curve on the projective plane that achieves $\cro([c],\Psi_H)$. By Proposition \ref{switch}, there exists a drawing $\Psi'_H$ of $G$ on the Euclidean plane such that
\begin{eqnarray*}
\cro (\Psi'_H) &=& \frac{\cro(c',\Psi_H)(\cro(c',\Psi_H) - 1)}{2}
\\ & < & \cro([c],\Psi_H)\times \cro([c],\Psi_K)
\\ & \le & \cro(\Psi).
\end{eqnarray*}

Let $\Psi'$ denote the drawing of $G$ on the projective plane obtained by disjoint union of the drawings $\Psi'_H$ and $\Psi_K$. The drawings $\Psi'_H$ and $\Psi_K$ do not cross each other and since $\Psi_K$ is an embedding, all crossings of $\Psi'$ are crossings of $\Psi'_H$. It follows that the drawing $\Psi'$ of $G$ on the projective plane has strictly fewer crossings than $\Psi$ and the restrictions to $H$ and $K$ do not cross.
\end{proof}

\subsection{Removing the two crosscaps of the Klein bottle}

\begin{prop}\label{twist}
Let $\Psi$ denote a drawing of a graph $G$ on the Klein bottle. Let $a'$ be a simple curve freely homotopic to $a$ and $m'$ a simple curve freely homotopic to $m$ such that neither $a'$ nor $m'$ contains any vertex of $G$, and such that $a'$ and $m'$ cross only once. Then there is a drawing $\Psi'$ of $G$ on the Euclidean plane such that
$$ \cro (\Psi') = \cro (\Psi) + \cro(a',\Psi) \times \cro(m',\Psi)+\frac{\cro(m',\Psi) (\cro(m',\Psi) - 1)}{2} .$$
\end{prop}

\begin{proof}
We cut the Klein bottle open along $m'$, disconnecting $\cro(m',\Psi)$ edges of $\Psi$. By definition of $m$, the resulting surface is a cylinder. We reconnect the cut edges such that their new part remains in a small neighborhood of $a'$, creating exactly $\cro(a',\Psi)$ crossings for each cut edge. Moreover, we can draw the $\cro(m',\Psi)$ edges so that they cross each other only once. We obtain a drawing of $G$ on the cylinder with the desired crossing number, therefore a drawing $\Psi '$ on the Euclidean plane with the same crossing number.
\end{proof}

\section{Maximum number of pairwise edge-disjoint one-sided circuits}
\label{Section:pro}

\begin{prop}
\label{pro1'}
Let $\Psi$ be an embedding of an eulerian graph on the Klein bottle. Then the maximum number of pairwise edge-disjoint one-sided circuits equals 
$$
\min(\cro([a],\Psi)+\cro([b],\Psi),\cro([m],\Psi)).
$$
Moreover, we can decompose $\Psi$ into a collection of circuits $I$ that achieves the maximum number of one-sided circuits and such that the number of $m$-circuits in $I$ is
$$
\frac{1}{2} \left( \cro([a],\Psi)+\cro([b],\Psi) - \min(\cro([a],\Psi)+\cro([b],\Psi),\cro([m],\Psi)) \right).
$$
\end{prop}
 
\begin{proof}
Let $\Psi$ be an embedding of an eulerian graph $G$ on the Klein bottle. Consider a collection $I$ of edge-disjoint one-sided circuits of $\Psi$. Every one-sided circuit intersects either $a$ or $b$. Consequently, for each circuit $c$ in $I$, $\cro([a],c) \ge 1$ or $\cro([b],c)\ge 1$. Hence,
$$
\cro([a],\Psi)+\cro([b],\Psi)\ge\sum_{c\in I}(\cro([a],[c])+\cro([b],[c]))
\ge |I|.
$$
Similarly, every one-sided circuit intersects $m$. Hence,
$$
\cro([m],\Psi)\ge\sum_{c\in I}\cro([m],[c]) \ge  |I|.
$$
Therefore the maximum number of pairwise edge-disjoint one-sided circuits is smaller than $\min(\cro([a],\Psi)+\cro([b],\Psi),\cro([m],\Psi))$.

To complete the proof of Proposition \ref{pro1'}, it remains to decompose $\Psi$ into a collection of circuits that contains 
$$
\min(\cro([a],\Psi) + \cro([b],\Psi) , \cro([m],\Psi))
$$ 
one-sided circuits and 
$$
\frac{1}{2}\cro([a],\Psi)+\cro([b],\Psi) - \min(\cro([a],\Psi)+\cro([b],\Psi),\cro([m],\Psi))
$$
$m$-circuits.

Let $I$ be a collection of circuits given by Theorem \ref{thm:de_graaf}, with $n_a$ $a$-circuits, $n_b$ $b$-circuits, $n_m$ $m$-circuits and $n_e$ $e$-circuits. By definition of $I$, the following equalities hold.  
\begin{equation}
\label{E}
\begin{array}{rcl}
\cro([a],\Psi)&=&n_a +n_m
\\ \cro([b],\Psi)&=&n_b+n_m
\\ \cro([m],\Psi) &=&n_a +n_b + 2n_e.
\end{array}
\end{equation}

If $n_m$ or $n_e$ equals zero, then the result follows. Now assume that $n_m$ and $n_e$ are positive.

Let $r=\min(n_m,n_e)$. Consider $r$ distinct $m$-circuits $m_1,\ldots,m_{r}$ and $r$ distinct $e$-circuits $e_1,\ldots,e_{r}$ in $I$. For every integer $i$ between $1$ and $r$, the circuits $m_i$ and $e_i$ intersect and can be decomposed into an $a$-circuit $a_i$ and an $b$-circuit $b_i$. Thus, we get $n_a+r$ $a$-circuits, $n_b+r$ $b$-circuits, $n_m - r$ $m$-circuits and $n_e - r$ $e$-circuits. The resulting collection of circuits $I'$ still decomposes $\Psi$. According to (\ref{E}),
$$
(n_a+r)+(n_b+r)=\min(\cro([a],\Psi) + \cro([b],\Psi) , \cro([m],\Psi)),
$$
and 
$$
2(n_m -r) = \cro([a],\Psi) + \cro([b],\Psi) - \min(\cro([a],\Psi) + \cro([b],\Psi) , \cro([m],\Psi)).
$$
Thus $I'$ is the desired collection of circuits.
\end{proof}

\section{Maximum number of edge-disjoint $a$-circuits}
\label{Section:pro2}

\begin{prop}
\label{pro2'}
  Let $\Psi$ be an embedding of an eulerian graph on the Klein bottle. Then the maximum number of edge-disjoint $a$-circuits equals 
$$
\min(\cro([a],\Psi),\cro([m],\Psi)).
$$
\end{prop}

\begin{proof}
Let $\Psi$ be an embedding of an eulerian graph on the Klein bottle. Consider a collection $I$ of edge-disjoint $a$-circuits of $\Psi$. Every $a$-circuit intersects $a$. Hence,
$$
\cro([a],\Psi)\ge\sum_{c\in I}\cro([a],c)\ge |I|.
$$
Similarly, every $a$-circuit intersects $m$. Hence,
$$
\cro([m],\Psi)\ge\sum_{c\in I}\cro([m],c)
 \ge |I|.
$$
Therefore the maximum number of pairwise edge-disjoint $a$-circuits is smaller than $\min(\cro([a],\Psi)+\cro([b],\Psi),\cro([m],\Psi))$.

To complete the proof of Proposition \ref{pro2'}, it remains to exhibit $\min(\cro([a],\Psi), \cro([m],\Psi))$ $a$-circuits.

Let $I$ be a collection of circuits of $\Psi$ as stated in Theorem \ref{thm:de_graaf}, with $n_a$ $a$-circuits, $n_b$ $b$-circuits, $n_m$ $m$-circuits and $n_e$ $e$-circuits. If $n_a$ or $n_m$ equals zero, then the result follows. Now assume that $n_a$ and $n_m$ are positive.

Let $r=\min(n_m,n_e)$. Consider $r$ distinct $m$-circuits $m_1,\ldots,m_{r}$ and $r$ distinct $e$-circuits $e_1,\ldots,e_{r}$ in $I$. For every integer $i$ between $1$ and $r$, the circuits $m_i$ and $e_i$ intersect and can be decomposed into an $a$-circuit $a_i$ and an $b$-circuit $b_i$. Consequently, we get $n_a+r$ $a$-circuits, $n_b+r$ $b$-circuits, $n_m-r$ $m$-circuits and $n_e-r$ $e$-circuits. The resulting collection of circuits $I'$ still decomposes $\Psi$.

Let $s=\min(n_m-r,n_b+r)$. If $n_m\le n_e$, then $s$ equals zero and we have found $\cro([a],\Psi)$ $a$-circuits. Otherwise, consider $s$ distinct $b$-circuits $b_1,\ldots,b_s$ and $s$ distinct $m$-circuits $m'_1,\ldots,m'_s$ of $I'$. For every integer $i$ between $1$ and $r$, the circuits $b_i$ and $m'_i$ intersect and can be decomposed into an $a$-circuit $a'_i$. And so, we get $n_a+r+s$ $a$-circuits.

According to (\ref{E}),
\begin{eqnarray*}
 n_a+r+m &=& n_a+r+\min(n_m-r,n_b+r)
\\ &=&n_a + \min(n_m,n_b + 2n_e)
\\ &=&\min(\cro([a],\Psi), \cro([m],\Psi)).
\end{eqnarray*}
Proposition \ref{pro2'} is proved.
\end{proof}

Note that Proposition \ref{pro2'} still holds when replacing $a$ by $b$.

\section{Main result}
\label{Section:proof}

This section is devoted to the proof of our main result. First, we need to prove the following special case of the theorem.

\begin{lem}
\label{lem:disjoint}
Let $G$ be the disjoint union of two eulerian connected graphs $H$ and $K$. If $G$ has a drawing on the Klein bottle such that the restrictions to $H$ and $K$ are embeddings that cross each other, then we can find another drawing of $G$ on the Klein bottle with strictly fewer crossings such that the restrictions to $H$ and $K$ do not cross. 
\end{lem}

\begin{proof}
Let $\Psi$ be a drawing of $G$ on the Klein bottle such that the restrictions $\Psi_H$ to $H$ and $\Psi_K$ to $K$ are embeddings. 

To prove Lemma \ref{lem:disjoint}, it is enough to find two drawings $\Psi_H'$ and $\Psi_K'$ of $H$ and $K$ on two disjoint subsurfaces of the Klein bottle such that the sum of the crossings of $\Psi'_H$ and the crossings of $\Psi'_K$ is strictly less than the crossings of $\Psi$. Indeed, let $\Psi_H'$ and $\Psi_K'$ be such drawings and let $\Psi'$ denote the drawing of $G$ on the Klein bottle plane obtained by disjoint union of the drawings $\Psi'_H$ and $\Psi_K'$. Then the number of crossings of $\Psi'$ is the sum of the crossings of $\Psi'_H$ and the crossings of $\Psi'_K$. It follows that the drawing $\Psi'$ of $G$ on the Klein bottle has strictly fewer crossings than $\Psi$ and the restrictions to $H$ and $K$ do not cross.

For convenience, we denote $\cro([a],\Psi_H)$, $\cro([b],\Psi_H)$ and $\cro([m],\Psi_H)$ by $h_{a}$, $h_{b}$ and $h_{m}$, respectively. Similarly, we denote $\cro([a],\Psi_K)$, $\cro([b],\Psi_K)$ and $\cro([m],\Psi_K)$ by $k_{a}$, $k_{b}$ and $k_{m}$, respectively. 

We can assume without loss of generality that $h_{m} \leq k_{m}$. 

By Proposition \ref{pro1}, there exist a decomposition of $\Psi_H$ into pairwise edge-disjoint circuits with $\min(h_{m} , h_{a} + h_{b}) $ one-sided circuits and $\left( h_{a} + h_{b} - \min (h_{m} , h_{a} + h_{b})\right) /2$ $m$-circuits. Each one-sided circuit crosses $\Psi_K$ at least $\min( k_{a} , k_{b})$ times, and each $m$-circuit crosses $\Psi_K$ at least $k_{m}$ times. Counting the crossings between $\Psi_H$ and $\Psi_K$ gives the following inequality.
\begin{equation}\label{Q1}
\cro (\Psi) \geq \min(h_{m} , h_{a} + h_{b}) \times \min(k_{a} ,  k_{b}) + \frac{1}{2} \left( h_{a} + h_{b} - \min(h_{m} , h_{a} + h_{b})\right) \times k_{m} .
\end{equation}
With a similar decomposition of $\Psi_K$ we obtain
\begin{equation}\label{Q2}
\cro(\Psi) \geq \min(k_{m} , k_{a} + k_{b}) \times \min(h_{a} ,  h_{b}) + \frac{1}{2} \left( k_{a} + k_{b} - \min(k_{m} , k_{a} + k_{b}) \right) \times h_{m} .
\end{equation}

Beside, by Proposition \ref{pro2}, there exist $\min( k_{a} , k_{m})$ pairwise edge-disjoint $a$-circuits of $\Psi_K$. Each of them crosses $\Psi_H$ at least $h_{a}$ times, therefore
\begin{equation}\label{Q3}
\cro (\Psi) \geq \min( k_{a} , k_{m} )  \times h_{a}.
\end{equation}
Similarly, considering $b$-circuits gives
\begin{equation}\label{Q4}
\cro (\Psi) \geq \min( k_{b} , k_{m} )  \times h_{b}.
\end{equation}

\noindent Let $m_1 \le m_2 \le m_3 \le m_4$ be an ordering of the numbers $h_{a}, h_{b} ,k_{a} , k_{b}$.

\noindent\textbf{(Case 1)} 
If $k_m\ge m_2$, then applying twice Proposition \ref{switch} provides a drawing $\Psi_H '$ of $H$ and a drawing $\Psi_K'$ of $K$ on disjoint subsurfaces of the Klein bottle such that
$$
\cro (\Psi_H ')+\cro(\Psi_K') = \frac{1}{2}m_1 \times (m_1 - 1) + \frac{1}{2}m_2 \times (m_2 - 1).
$$
By definition of $m_2$ and since $k_m\ge m_2$,
$$
m_2 \times m_2 \le \max(\min( k_{a} , k_{m} ) h_{a},\min( k_{b} , k_{m} ) h_{b}).
$$
Hence, by (\ref{Q3}) and (\ref{Q4}),
$$
\cro (\Psi_H ')+\cro(\Psi_K') < m_2\times m_2 \le \cro (\Psi).
$$

\medskip

\noindent\textbf{(Case 2)}
If $k_{m}<m_2$, then 
$$
h_m\le k_m \le m_1+m_2 \le h_{a}+h_{b}.
$$
Thus, (\ref{Q1}) becomes
\begin{equation}\label{Q'1}
\cro (\Psi) \geq h_{m} \times \min(k_{a} ,  k_{b}) + \frac{1}{2}( h_{a} + h_{b} - h_{m} ) \times k_{m}.
\end{equation}
Since $k_{m}\le k_{a}+k_{b}$, the (\ref{Q2}) becomes
\begin{equation}\label{Q'2}
\cro (\Psi) \geq k_{m} \times h_{a} + \frac{1}{2} ( k_{a} + k_{b} - k_{m} ) \times h_{m}.
\end{equation}

\medskip

\noindent\textbf{(Case 2.1)}
If $h_m \le k_{a} + k_{b} - k_{m}$, then by Proposition \ref{twist} there exists a drawing $\Psi_H '$ of $H$ on the Euclidean plane such that
$$
\cro (\Psi_H ') = h_{m} \times h_{a}+\frac{1}{2}h_{m} \times (h_{m} - 1).
$$

We get by (\ref{Q'2})
\begin{align*}
\cro (\Psi_H ') + \cro (\Psi_K) & = h_{m} \times h_{a}+\frac{1}{2}h_{m} \times (h_{m} - 1)   \\
 & \leq k_{m} \times h_{a}+\frac{1}{2}(k_{a} + k_{b} - k_{m}) \times (h_{m} - 1)  \\
 & < \cro (\Psi).
\end{align*}

\medskip

\noindent\textbf{(Case 2.2)}
If $h_{m} > k_{a} + k_{b} - k_{m}$, then $k_{m} < m_2$ implies
$$
h_{m}+\max(k_{a},k_{b}) \ge h_{m}+k_{m} > k_{a}+k_{b}.
$$
Hence $h_{m}>\min(k_{a},k_{b})=m_1$. 

\medskip

\noindent\textbf{(Case 2.2.1)}
If $m_{1} < k_{m}/2$ then we apply Proposition \ref{twist}. There exists a drawing $\Psi_K '$ of $K$ on the Euclidean plane such that
$$
\cro (\Psi_K ') = k_{m} \times m_{1} + \frac{1}{2}k_{m} \times (k_{m} - 1).
$$
By (\ref{Q'1}),
\begin{eqnarray*}
\cro(\Psi_H) +\cro (\Psi_K ') & \le & k_{m} \times m_{1} + \frac{1}{2}k_{m} \times (k_{m} - 1) \\
& \leq & h_{m} \times m_{1} + (k_{m} - h_{m}) \times m_{1} + \frac{1}{2}k_{m} \times (k_{m} - 1) \\
& < & h_{m} \times m_{1} + (2  k_{m} - h_{m}) \times \frac{1}{2}k_{m} \\
& < & h_{m} \times \min(k_{a},k_{b}) + \frac{1}{2}( h_{a} + h_{b} - h_{m} ) \times k_{m} \\
& < & \cro(\Psi).
\end{eqnarray*}

\medskip

\noindent\textbf{(Case 2.2.2)}
If $m_1 \geq k_{m}/2$ and $m_2 < 2 k_{m}$, then applying twice Proposition \ref{switch} provides a drawing $\Psi_H '$ of $H$ and a drawing $\Psi_K'$ of $K$ on disjoint subsurfaces of the Klein bottle such that
$$
\cro (\Psi_H ')+\cro(\Psi_K') = \frac{1}{2}m_1 \times (m_1 - 1) + \frac{1}{2}m_2 \times (m_2 - 1).
$$
Since 
$$
m_{1}<h_{m}, \quad m_{2}>k_{m}, \quad m_{2}\le h_{a} \quad \textrm{and} \quad m_{1}+m_{2}\le k_{a}+k_{b},
$$ 
we get, by (\ref{Q'2}), 
\begin{eqnarray*}
\cro (\Psi_H') + \cro (\Psi_K ') & = &  \frac{1}{2}m_{2} \times (m_{2} - 1) + \frac{1}{2}m_{1} \times (m_{1} - 1)  \\
& < & \frac{1}{2}m_{2} \times 2k_{m} +\frac{1}{2}m_{1} \times (h_{m}-1) \\
& < &  h_{a} \times k_{m} + \frac{1}{2}(m_{1}+m_{2}-k_{m})\times h_{m} \\
& < & h_{a} \times k_{m} +\frac{1}{2}(k_{a}+k_{b}-k_{m})\times h_{m} \\
& < & \cro (\Psi).
\end{eqnarray*}

\medskip

\noindent\textbf{(Case 2.2.3)}
If $m_{1} \geq k_{m}/2$ and $m_{2} \geq 2 k_{m}$ then we apply Proposition \ref{twist}. There exists a drawing $\Psi_K '$ of $K$ on the Euclidean plane such that
$$
\cro (\Psi_K ') = k_{m} \times m_{1} + \frac{1}{2}k_{m} \times (k_{m} - 1).
$$
Hence, by (\ref{Q3}),
\begin{eqnarray*}
\cro (\Psi_H) + \cro (\Psi_K ') & = &  k_{m} \times m_{1} + \frac{1}{2}k_{m} \times (k_{m} - 1) \\
& < & k_{m}\times \left(m_{1}+\frac{1}{2}k_{m}\right) \\
& < & k_{m} \times (2m_{1}) \\
& < & m_{1}\times m_{2} \\
& < & \cro (\Psi).
\end{eqnarray*}
\end{proof}

Now, we may prove the Theorem.

\begin{thm}
  Let $G$ be the disjoint union of two connected graphs $H$ and $K$. For every optimal drawing of $G$ on the Klein bottle, the restrictions to $H$ and $K$ do not cross.
\end{thm}

\begin{proof}
Let $G$ be the disjoint union of two connected graphs $H$ and $K$. Let $\Psi$ be an optimal drawing on the Klein bottle of $G$. 

First, assume that the restrictions to $H$ and $K$ are embeddings. Duplicate each edge of $G$ and denote by $G'$, $G'_1$, $G'_2$ the resulting eulerian graphs and by $\Psi'$ the resulting drawing. The drawing $\Psi'$ has $4\cro(\Psi)$ crossings. Since $G'$ is eulerian, by Proposition \ref{lem:disjoint}, we can find a drawing $\Psi''$ of $G'$, where the graphs $H'$ and $K'$ do not cross each other, which has strictly less than $4\cro(\Psi)$ crossings. Moreover, we can assume that every pair of parallel edges are drawn close enough to have the same crossings. Therefore, two pairs of parallel edges intersect each other either four times or do not. Thus, by deleting one copy of each edge, we get a drawing of $G$ with strictly less than $\cro(\Psi)$ crossings such that the drawings of $H$ and $K$ do not intersect. 

Secondly, suppose that the restrictions to $H$ and $K$ are not embeddings. Consider the graphs $G''_1$ and $G''_2$ obtained from $H$ and $K$ by adding a vertex for each internal crossing. The corresponding drawings are embeddings and we can now apply what was shown just above. Theorem \ref{thm:main} is proved when we replace the new vertices by the former crossings. 
\end{proof}

\begin{rem}
Applying the same arguments of the proof of Theorem \ref{thm:main} and using Corollary \ref{cor:planproj} instead of Proposition \ref{lem:disjoint}, we can prove that Conjecture \ref{conj} holds for the projective plane.
\end{rem}


\begin{thebibliography}{00}

\bibitem{brahana_23}
 H. R. Brahana, Systems of circuits of two-dimensional manifolds, \emph{Ann. of Math.}, {\bf 30}(1923), 234--243

\bibitem {bollobas_98}
  B. Bollob\'as, Modern Graph Theory, \emph{Springer, New-York}, 1998
  
\bibitem{devos_08}
  M. DeVos, B. Mohar and R. Samal, \emph{Open Problems Garden},\\
  {\tt http://garden.irmacs.sfu.ca/?q=op/drawing\_disconnected\_graphs\_on\_surfaces}, 2007  

\bibitem{erdos_73}
  P. Erd\"os and R. K. Guy, Crossing number problems, \emph{The American Mathematical Monthly}, {\bf 80}(1973), 52--57

\bibitem{toth_00}
 J. Pach and .Toth, Thirteen problems on crossing numbers, \emph{Geombinatorics}, {\bf 9}(2000), 195--207

\bibitem{mohar_01}
  B. Mohar and C. Thomassen, Graphs on surfaces, \emph{Johns Hopkins University Press, Baltimore}, 2001

\bibitem{stillwell_80}
  J. Stillwell, Classical topology and combinatorial group theory, \emph{Springer-Verlag, New York}, 1980, 62--67

\bibitem{negami_97}
  S. Kawrencenko and S. Negami, Irreducible triangulations of the Klein Bottle, \emph{J. Combin. Theory Ser. B}, {\bf 70} (1997), 265--291, doi:10.1006/jctb.1997.9999

\bibitem{de_graaf_97} 
  M. de Graaf and A. Schrijver, Decomposition on surfaces, \emph{J. Combin. Theory Ser. B}, {\bf 70}(1997), 157--165

\bibitem{lins_81}
  S. Lins, A minimax theorem on circuits in projective graphs, \emph{J. Combin. Theory Ser. B}, {\bf 30}(1981), 253--262, doi:10.1016/0095-8956(81)90042-3

\bibitem{schrijver_89}
  A. Schrijver, The Klein bottle and multicommodity flows, \emph{ Combinatorica}, {\bf 9}(1989), 375--384, doi:10.1007/BF02125349

\bibitem{schrijver_91}
  A. Schrijver, Decomposition of graphs on surfaces and a homotopic circulation theorem, \emph{J. Combin. Theory Ser. B}, {\bf 51}(1991), 161--210

\bibitem{luo_97}
 F. Luo, On non-separating simple closed curves in a compact surface, \emph{Topology}, {\bf 36}(2)(1997), 381--410

\end{thebibliography}
\end{document}